\documentclass[a4paper, 11pt]{article}
\usepackage{anysize}
\marginsize{3,5cm}{2,5cm}{2,5cm}{2,5cm}
\usepackage{amssymb}
\usepackage{amsmath,amsbsy,amssymb,amscd}
\usepackage{t1enc}
\usepackage[cp1250]{inputenc}
\usepackage[british]{babel}
\usepackage[all]{xy}
\usepackage{color}
\usepackage{amsfonts}
\usepackage{latexsym}
\usepackage{amsthm}
\usepackage{mathrsfs}
\usepackage{hyperref}
\usepackage{graphicx}
\usepackage{comment}

\newtheorem{conjecture}{Conjecture}

\usepackage{color}
\usepackage{caption}
\usepackage{subcaption}
\usepackage{enumerate}
\usepackage{xcolor}
\usepackage[draft]{todonotes}

\definecolor{orange}{rgb}{1,0.5,0}

\usepackage{mathrsfs} 

\DeclareMathAlphabet{\mathpzc}{OT1}{pzc}{L}{it} 

\theoremstyle{definition}
\newtheorem{definition}{Definition}[section]
\newtheorem{theorem}[definition]{Theorem}

\newtheorem{proposition}[definition]{Proposition}

\newtheorem{lemma}[definition]{Lemma}

\newtheorem{remark}[definition]{Remark}

\def\Re{\mathrm{Re\,}}

\def\geq{\geqslant}
\def\leq{\leqslant}
\def\R{\mathbb{R}}

\def\Z{\mathbb{Z}}
\def\N{\mathbb{N}}

\def\epsilon{\varepsilon}

\newcommand{\beq}{\begin{equation}}
\newcommand{\eeq}{\end{equation}}
\newcommand{\bea}{\begin{eqnarray}}
  \newcommand{\eea}{\end{eqnarray}}
  \newcommand{\beab}{\begin{eqnarray*}}
  \newcommand{\eeab}{\end{eqnarray*}}

  \newcommand{\be}{\begin{equation}}
  \newcommand{\ee}{\end{equation}}

\title{Density of orbits of horocycle flows at sub-quadratic polynomial times}
\author{Adam Kanigowski and Maksym Radziwi\l\l}
\begin{document}
\maketitle
\begin{abstract}Let $\Gamma\subset PSL(2,\R)$ be such that the space $X=\Gamma\slash PSL(2,\R)$ is not compact. Let 
$(h_t)$ be the horocycle flow acting on $X$. We show that for every $x\in X$ that is not periodic for $(h_t)$ and for every $\delta\in (0,1)$ the orbit $\{h_{n^{2-\delta}}x\}_{n\in \N}$ is dense in $X$. Assuming additionally the Hardy-Littlewood conjecture we show that for every non-periodic $x\in X$, $\{h_{p}x\}_{p-  \text{prime}}$ is dense in $X$. Finally we show that  for $\Gamma=PSL(2,\Z)$, $\{h_{n^2}y_q\}_{n<q}$ equidistribute, as $q\to \infty$ along primes congruent to $1 \pmod{4}$, towards Haar measure, where $\{y_q\}$ is a sequence of periodic points of period $q$.
\end{abstract}
\section{Introduction}
The paper is concerned with studying the orbits of horocycle flows at polynomial times. Recall that horocycle flows acting on quotients of $SL(2,\R)$ are a fundamental object in ergodic theory, dynamics and number theory. H.\ Furstenberg \cite{Fur} showed that if a lattice $\Gamma\subset SL(2,\R)$ is cocompact, then the horocycle flow on $X=\Gamma\slash SL(2,\R)$ is uniquely ergodic. If the lattice $\Gamma$ is not cocompact then the flow is not uniquely ergodic (due to presence of periodic orbits) but a full classification of invariant measures was given by Dani \cite{Dani} and Dani-Smillie, \cite{DS}. It follows that the only invariant measures are either the Haar measure or a measure supported on a periodic orbit. Horocycle flows enjoy quantitative ergodicity, i.e. one gets (polynomial) rates of deviation of ergodic averages as shown in the work of Flaminio-Forni, \cite{Flaminio-Forni} and Str\"ombergsson \cite{Strombergsson}. Horocycle flows are  polynomially mixing \cite{Howe,Moore} and in fact have countable Lebesgue spectrum, \cite{Par}. Due to the algebraic nature (rotation on a Lie group) it is natural to study horocycle flows at sparse times. Indeed, Margulis and Shah conjectured that the orbits of horocycle flows equidistribute when sampled at polynomial (or prime) times. This conjecture is still widely open in full generality although some progress towards it was made in recent years. Venkatesh, \cite{Ven}, showed that horocycle flows equidsitribute when sampled at times $\{n^{1+\eta}\}$ for some small $\eta=\eta_\Gamma>0$. Some improvements to Venkatesh's result were obtained in \cite{FFT} where the authors showed some uniform bounds on $\eta$ (but still could get up to  $\eta<1/12$). This general approach is based on studying so called twisted integrals using quantitative mixing and quantitative ergodicity. Sarnak and Ubis \cite{SU}, studied the orbits of horocycle flows at prime times and proved that prime orbits are not to concentrated -- they enter every set of sufficiently large diameter (see also \cite{Stre} for more general lattices). Generalizing  their approach, McAdam \cite{McAdam} showed that the orbits of horocycle flows become dense when sampled at  $k$-primes (integers having at most $k$ prime factors). More recently in \cite{FKR} the authors showed that the orbits of horocycle flows equidistribute when sampled at $2$-primes (for arithmetic lattices). Their methods used a recent breaktrough in \cite{LMW} where the authors obtained quantitative ergodicity for the square of the horocycle flow.  

In this paper we study orbits of horocycle flows at times $\{n^{2-\delta}\}$ for $\delta>0$. As mentioned above the method of Venkatesh does not work for times $\{n^{1+\eta}\}$, with $\eta>1/10$. The main result of this paper (see Theorem \ref{thm:main}) is that for non-cocompact lattices and for any non-periodic point $x\in X$ the orbit $\{h_{n^{2-\delta}} x\}$ is dense in $X$. We use a very different approach based on approximating horocycle orbits by periodic orbits on short intervals (we provide more details in the outline below). This in spirit is more similar to \cite{SU}, however we need a finer approximation but on a shorter interval of times. In particular our method seems to be restricted to the non-compact case, as otherwise there are no periodic orbits. Using  similar methods we can also show that prime orbits are dense conditionally on the Hardy-Littlewood conjecture. Finally recall that by a result of Sarnak, \cite{Sarnak}, the sequence of probability measures on periodic orbits whose lengths are tending to $\infty$ converges to the Haar measure. We show an analogous result for squares, i.e. we show that  $\{h_{n^2}y_q\}_{n<q}$ equidistribute, as $q\to \infty$, towards Haar measure, where $\{y_q\}$ is a sequence of periodic points of period $q$. This uses the spectral theory of automorphic forms to reduce the problem to subconvex estimates for $L(s, f \otimes \chi_d)$ and $L(s, \chi_d)$, due to respectively Iwaniec and Burgess. 

\section{Main Results}
Let $\Gamma\subset PSL(2,\R)$ be such that the space $X=\Gamma\slash PSL(2,\R)$ is not compact. Consider the horocycle flow $h_t=\begin{pmatrix}1&t\\0&1 \end{pmatrix}$ given by 
\begin{equation}\label{eq:hor}
h_t(\Gamma x)=\Gamma xh_t.
\end{equation}
Our main result is the following: 
\begin{theorem}\label{thm:main} For every $\delta\in (0,1)$  and every non-periodic $x\in X$, the orbit $\{h_{n^{2-\delta}}x\}$ is dense in $X$.
\end{theorem}
We also study density of horocycle orbits when sampled at prime times. Assuming the Hardy-Littlewood conjecture we show the following:
\begin{theorem}\label{thm:main'} Assume the Hardy-Littlewood conjecture (specifically the conjecture denoted by $\mathcal{C}_{10^6}(N^{1/50}, N)$ in Section \ref{se:numbertheory}). Then for every non-periodic $x\in X$, the orbit $\{h_px\}_{p- \text{prime}}$ is dense in $X$.
\end{theorem}
Finally we study distribution of the sequence of periodic orbits when sampled at the squares in the special case $\Gamma=PSL(2,\Z)$. We show the following:

\begin{theorem}\label{thm:main''}
Let $\{y_q\}_{q\in \N}$ be a sequence of $(h_t)$- periodic points of period $q$. We show that the sequence of measures 
$$
\mu_q:=\frac{1}{q}\sum_{n<q} \delta_{h_{n^2}(y_q)}
$$
converges to the Haar measure on $X=PSL(2,\Z)\slash PSL(2,\R)$ as $q\to \infty$ along primes congruent to $1 \pmod{4}$. 
\end{theorem}
\begin{remark}The proof of this result can be extended to other moduli, we focus on the case of $q \equiv 1 \pmod{4}$ as it captures the main ideas. 
\end{remark}

{\bf Outline of the proof:} The main idea behind Theorems \ref{thm:main} and \ref{thm:main'} is similar. The first key observation is that if $\{w_R\}_{R\in \N}$ is a  sequence of periodic points of period going to $\infty$ (and irrational), then $\{h_{n^{2-\delta}}w_R\}_{n\in \N}$ is $\epsilon$-dense provided that $R$ is large enough in terms of $\epsilon$. Using quantitative equidistribution results for horocyle flows we show that if $x$ is a non-periodic point then, for a sequence of times $\{T_i\}$ a long enough piece (of length $T_i^{1/2-o(1)}$) of the orbit $\{h_tx\}_{t<T_i}$ can be approximated by a periodic orbit $Orb(w)$ of a fixed period $R$, where $R$ is fixed at this stage, but in the end we let $R\to \infty$. The approximation on the interval $I$ of length $T_i^{1/2-o(1)}$  is a  sheared approximation, i.e. it is {\bf not} true that $h_tx$ is close to $h_tw$ for $t\in I$, rather there is a function $m(t)=m_{x,w}(t)$ such that $h_tx$ is close to $h_{m(t)}w$. This approximation, in rough terms, reduces the problem to studying equidistribution of $\{m(a_n)\alpha\}_{n\in I}$, where $(a_n)$ is either $n^{2-\delta}$ or $(a_n)$ is the characteristic function of primes and $\alpha=(per(w))^{-1}$. Note that this in general is a difficult problem (especially for primes) as the interval $I$ is of length $o(T^{1/2})$.  The main idea behind Theorem \ref{thm:main} is that one can approximate the function $m(\cdot)$ on $I$ by a (high degree) polynomial and then use Weyl equidistribution theorem on exponential sums at polynomial phases. The case of primes is much more difficult as the interval under consideration is to short in general. We use a different argument here. More precisely, by further splitting the interval $I$ into intervals $J$ of size $|I|^{\eta}$ we can indeed approximate $h_tx$ by $h_t(w_J)$ where $w_J\in Orb(w)$. Since $w$ is periodic, the $w_J$ will not matter in the end as it is just a constant shift in the phase. The extra idea is to average the problem along  possible periods of periodic points. More precisely assuming the Hardy-Littlewood conjecture we show that there exists $q\sim |J|^{1/4}$ such that for 'most' $a\sim q$,$a<q$, $(a,q)=1$, we have for $H=|J|$
\begin{equation}\label{eq:mh}
\Big|\sum_{n\in [M,M+H]} e(n\frac{a}{q})\Lambda(n)\Big|=o(H),
\end{equation}
for all but a set $T^{1/2-\epsilon}$ of $M\in [0,T]$. We then take the orbit $w$ to have period $q/a$. In particular, since $|I|\sim T^{1/2-o(1)}$ it follows that there is a $M\in I$, such that the interval $[M,M+H]$ satisfies \eqref{eq:mh}. We then use the fact that on $[M,M+H]$, $h_tx$ is close to $h_t(w_J)$, and $\{h_t(w_J)\}$ equdisitributes to the periodic measure $\mu_w$ due to \eqref{eq:mh}. Finally the periodic measure itself equidistributes towards Haar measure as the period goes to $\infty$. 

For the proof of Theorem \ref{thm:main''} we use the spectral theory of automorphic forms to reduce  the problem to averaging cusp forms and Eisenstein series for $\text{SL}_2(\mathbb{Z})$ over a series of points 
$$
\frac{n^2}{q} + \frac{i}{q}.
$$
We then expand the cusp forms and Eisenstein series into a Fourier series, the averaging over $n^2 / q$ then amounts to twisting this Fourier series by a quadratic Gauss sum. Since the Gauss sum can be directly evaluated in terms of quadratic characters, we reduce ourselves to the problem of bounding the Fourier expansion of a cusp form twisted by a quadratic character of conductor $q$. In the case of cusp forms, this is equivalent to obtaining a subconvex bound for $L(\tfrac 12 + it, f \otimes \chi_d)$ in the $d$-aspect, a result of Iwaniec, and for Eisenstein series this reduces to subconvexity for $L(s, \chi_d)$, a classical result of Burgess.

\section{Quantitative equidistribution for horocycle flows}

Let 
$$
h_t=\begin{pmatrix}1&t\\0&1 \end{pmatrix},\;\;\;\; g_s=\begin{pmatrix}e^{s/2}&0\\0&e^{-s/2} \end{pmatrix}, \;\;\;\;\; v_t=\begin{pmatrix}1&0\\t&1 \end{pmatrix}
$$
be the horocycle, geodesic and complementary horocycle flow respectively. The flows $(g_s)$ and $(v_t)$ also act on $X$ analogously to \eqref{eq:hor}.
 We recall the following renormalization relations
\begin{equation}\label{eq:norm}
h_t \cdot g_s=g_s\cdot  h_{e^{-s}t}\;\;\;\text{ and }\;\;\; v_t\cdot g_s=g_s\cdot v_{e^{s}t}, \;\; \text{for all}\;\;\; s,t\in \R.
\end{equation}

Let $\mu_{Haar}$ denote the Haar measure on $X$ and for $x\in X$ let $dist(x,e)$ denote the hyperbolic distance from $x$ to $e$.  The following quantitative equidistribution result was proven in Theorem 1.1. \cite{Strombergsson} (see also Theorem 5.14 in \cite{Flaminio-Forni}):
\begin{theorem}\label{thm:SFF} There exists $\alpha\in (-0,1)$ such that for every $f\in C^4(X)$ of compact support and for all $x\in X$, 
$$
\frac{1}{T}\int_0^T f(h_t(x))dt=\mu_{Haar}(f)+ {\rm O}(\|f\|_{W_4})r^{-\alpha}$$
where $\|f\|_{W_4}$ denotes the Sobolev norm and $r=r(x,T)=T\cdot e^{-dist(g_{\log T}(x),e)}$.
\end{theorem}

We have also an analogous result for the flow $(v_t)$ that we state in the remark below for future reference:
\begin{remark}\label{rem:SFF} Let $f\in C^4(X)$  be of compact support. Then for all $x\in X$, 
$$
\frac{1}{T}\int_0^T f(v_t(x))dt=\mu_{Haar}(f)+ {\rm O}(\|f\|_{W_4})(\bar{r}^{-\alpha}+T^{-\alpha})
$$
where $\bar{r}=\bar{r}(x,T)=T\cdot e^{-dist(g_{-\log T}(x),e)}$.
\end{remark}

\begin{lemma}\label{lem:perp} Let $p\in X$ be a periodic point for $(h_t)$ of period $R>1$ and $\delta\in (0,1)$. Then for every $T\geq T_{\delta,R}$ the following holds: there exists $c\in \R_+$, $c\sim  T^{-1/2+\delta/12}$ such that 
$$
g_{\log T}(v_c(p))\in B_X(e,1):=\{x\in X\;:\; d_X(x,e)\leq 1\}.
$$
\end{lemma}
\begin{proof}  Note that by \eqref{eq:norm}, 
$$
g_{\log T}\Big(\{v_{c}(p)\;:\; c\sim  T^{-1/2+\delta/12}\}\Big)= \{v_{r}(g_{\log T}(p))\;:\; r\sim  T^{1/2+\delta/12}\}.
$$

For a given $\epsilon>0$ we approximate $\chi_{B(e,1)}$ by two smooth functions $f_-\leq \chi_{B(e,1)}\leq f_+$, so that $f_-$ is supported inside $B(e,1)$ and $f_+$ is supported in $V_\epsilon(B(e,1))$ and such that $\|f_\pm-\chi_{B(e,1)}\|_{L^1}\leq \epsilon$. Moreover the size of derivatives of $f_{\pm}$ depends on $B(e,1)$ and $\epsilon>0$ but not on $T$.  Using this, by Remark \ref{rem:SFF}, we get
$$
\frac{1}{T^{1/2+\delta/12}}\int_{0}^{T^{1/2}+\delta/12} f_\pm(v_r (g_{\log T}(p))) dr= \mu_{Haar}(f_\pm) +{\rm O}(\|f_{\pm}\|_{W^4})\bar{r}^{-\alpha},
$$
where $\bar{r}=T^{1/2+\delta/12} \cdot e^{-dist(g_{-\log (T^{1/2+\delta/12})}(g_{\log T}(p)),e)}=  T^{1/2+\delta/12} \cdot e^{-dist(g_{\log(T^{1/2-\delta/12})}(p),e)}$.

Since $p$ is periodic of period $R$ it follows that there exists $i\leq k$ (where $k$ is the number of cusps) such that $p=g_{\log R}(h_z(e_i))$, where $|z|\leq 1$.  Therefore for some global $C>0$,

$$
dist(g_{\log(T^{1/2-\delta/12})}(p),e)=dist(g_{R+\log(T^{1/2-\delta/12})}(h_z(e_i)),e)\leq$$$$
 dist(g_{R+\log(T^{1/2-\delta/12})}(h_z(e_i)),h_z(e_i))+dist(h_z(e_i),e) \leq
R+\log(T^{1/2-\delta/12})+C
$$
if $T$ is large enough in terms of $R$. Therefore $\bar{r}\geq T^{1/2+\delta/12} \cdot T^{-1/2+\delta/12}\cdot e^{-C-R}=T^{\delta/6}\cdot e^{-C-R}. $

Since $\|f_\pm\|_{W^4}$ depends only on $B(e,1)$ and $\epsilon$, we can make $\mu_{Haar}(f)$ larger then ${\rm O}(\|f_{\pm}\|_{W^4})\bar{r}^{-\alpha}$, by taking $T$ large enough in terms of $\epsilon, R,C$.
 It then follows that 
$$
\int_{T^{1/2+\delta/12}/2}^{T^{1/2+\delta/12}} f(v_r g_{\log T}(p)) dr=$$$$
\int_{0}^{T^{1/2+\delta/12}} f(v_r g_{\log T}(p)) dr- \int_{0}^{T^{1/2+\delta/12}/2} f(v_r g_{\log T}(p)) dr\geq $$$$
\mu_{Haar}(f_-)T^{1/2+\delta/12}- \frac{1}{2}\mu_{Haar}(f_+)T^{1/2+\delta/12}+ o_{T\to\infty}(T^{1/2+\delta/12})>0,
$$
where in the last inequality we use that $ \mu_{Haar}(f_-)\geq \frac{2}{3} \mu_{Haar}(f_+)$.
 In particular, there exists $r\sim T^{1/2+\delta/12}$ such that  $v_{r}g_{\log T}(p)\in B(e,1)$. This finishes the proof.
\end{proof}

\section{Approximating pieces of orbits by periodic orbits.}
By Proposition 1.1. in  \cite{Ratner} it follows that a point $x\in X$ is non-periodic for $(h_t)$ if and only if the orbit $\{g_sx\}_{s>0}$ is recurrent, i.e. there exists a compact set $K\subset X$ and a sequence of times $\{T^x_i\}_{i=1}^\infty$ such that $g_{\log T^x_i}x\in K$. Since the initial point $x\in X$ is fixed in this section, we will drop the superscript $x$ from the return times $\{T^x_i\}$.  Moreover just by renaming the return times, we will suppose they are of the form $T_i\log T_i$, i.e. 
\begin{equation}\label{eq:ret}
g_{\log(T_i\log T_i)}x\in K
\end{equation}
In the following lemma we will approximate a piece of the orbit of $x$ by a periodic orbit.

\begin{proposition}\label{prop:per} Let $p\in X$ be a periodic point of period $R$. There is $i_0=i(R)\in \R_+$ such that for every $i\geq i_0$, there exists $t_i\in [T_i/2,T_i\log T_i]$  and $\gamma\in \Gamma$ such that 
\begin{equation}\label{eq:per}
\gamma x \begin{pmatrix}1&t_i\\0&1 \end{pmatrix} =p\begin{pmatrix}a&0\\c&a^{-1} \end{pmatrix},
\end{equation}
where $|a-1|<T^{-\delta\cdot 10^{-6}}$ and $c\cdot T_i^{1/2-\delta/10}\in [\frac{1}{4},4]$.
\end{proposition}
\begin{proof}
Let $V_{T}(p)=\Big\{\Gamma p\begin{pmatrix}a&0\\c&a^{-1} \end{pmatrix}\;:\; |a-1|<T^{-\delta/1000} \text{ and }c\cdot T^{1/2-\delta/10}\in [\frac{1}{4},4]\Big\}$. Note that the proposition is now equivalent to showing that for every non-periodic $x\in X$, there is $i_R>0$ such that for every $i\geq i_R$
there exists $t_i\in [T_i/2,T_i\log T_i]$ such that $h_{t_i}(x)\in V_{T_i}(p)$. For simplicity we denote $T_i$ by $T$. Let $A=60-5\delta$, so that $(1-\delta/A)(-1/2+\delta/12)=-1/2+\delta/10$. We will apply Lemma \ref{lem:perp} for the time $T=T^{1-\delta/A}$. It implies that there exists $s\sim T^{(1-\delta/A)(-1/2+\delta/12)}=T^{-1/2+\delta/10}$ such that $g_{\log(T^{1-\delta/A})}(v_s(p))\in B(e,1)$. Notice that $\{g_bv_{s+d}(p))\;:\; |b|<T^{-\delta\cdot 10^{-6}}, |d|\leq \frac{1}{T^{1-\delta/A}}\}\subset V_{T}(p)$. This just follows from the fact that 
$$
\begin{pmatrix}1&0\\s+d&1 \end{pmatrix}\cdot \begin{pmatrix}e^{b/2}&0\\0&e^{-b/2}\end{pmatrix}= \begin{pmatrix}e^{b/2}&0\\(s+d)e^{-b/2}&e^{-b/2}\end{pmatrix},
$$
and $(s+d)e^{-b/2}\cdot T^{1/2-\delta/10}\in [\frac{1}{3},3]$. So it is enough to  show that there exists 
$t_i\in [T/2,T\log T]$ such that $h_{t_i}(x)\in \{g_b(v_dv_s(p))\;:\; |b|<T^{-\delta\cdot 10^{-6}}, |d|\leq \frac{1}{T^{1-\delta/A}}\}$. 

By composing with $g_{\log(T^{1-\delta/A})}$ and using the renormalization equations \eqref{eq:norm}, this in turn is equivalent to showing that there exists $\bar{t}\in [T^{\delta/A}/2,T^{\delta/A}\log T]$ such that $h_{\bar{t}}(g_{\log(T^{1-\delta/A})}(x))\in N(s,p,T):=\{g_bv_d(g_{\log(T^{1-\delta/A})}v_sp)\;:\; |b|<T^{-\delta\cdot 10^{-6}}, |d|\leq 1 \}$. In fact we will show a stronger statement where we show it holds with $\tilde{B}=B(s,p,T)=B(g_{\log(T^{1-\delta/A})}v_sp,\frac{1}{4}T^{-\delta\cdot 10^{-6}})$ instead of $N(s,p,T)$.Notice that $\mu_{Haar}(\tilde{B})\sim_C T^{-3\delta \cdot 10^{-6}}$. In particular, we need to show that 
\begin{equation}\label{eq:chit}
\int^{T^{\delta/A}\log T}_{T^{\delta/A}/2} \chi_{\tilde{B}}(h_t(g_{\log(T^{1-\delta/A})}(x)))dt>0.
\end{equation}
Using Theorem \ref{thm:SFF} and approximating $\chi_{\tilde{B}}$ by smooth functions it follows  that
$$
\int^{T^{\delta/A}\log T}_{0} \chi_{\tilde{B}}(h_t(g_{\log(T^{1-\delta/A})}(x)))dt=$$$$
 T^{\delta/A}\log T \mu(\tilde{B})+ {\rm O}(T^{15\delta\cdot 10^{-6}})(T^{\delta/A}\log T)r^{-\alpha},
$$
where $r=(T^{\delta/A}\log T)\cdot e^{-dist(g_{\log(T^{\delta/A} \log T)}(g_{\log(T^{1-\delta/A})}(x)),e)}$. Note that by \eqref{eq:ret}, 
$$g_{\log(T^{\delta/A} \log T)}(g_{\log(T^{1-\delta/A})}(x)=g_{\log(T\log T)}x\in K.$$
 This implies that $r\geq C_K (T^{\delta/A}\log T)$ and this implies that $ {\rm O}(T^{15\delta\cdot 10^{-6}})(T^{\delta/A}\log T)\bar{r}^{-\alpha}= o(T^{\delta/A}\mu(\tilde{B}))$.
 Analogously, 
$$
\int^{T^{\delta/A}/2}_{0} \chi_{\tilde{B}}(h_t(g_{\log(T^{1-\delta/A})}(x)))dt= \frac{1}{2}T^{\delta/A}\mu(\tilde{B})+ {\rm O}(T^{15\delta\cdot 10^{-6}})(\frac{T^{\delta}}{2})\tilde{r}^{-\alpha},
$$
 where $\tilde{r}=(T^{\delta/A}/2)\cdot e^{-dist(g_{\log(T^{\delta/A}/2)}(g_{\log(T^{1-\delta/A})}(x)),e}$. Note that 
$ g_{\log(T^{\delta/A}/2)}(g_{\log(T^{1-\delta/A})}(x)= g_{-\log 2- \log\log (T)}g_{\log(T\log T)}x$ and $g_{\log(T\log T)}x\in K$. From this it follows that $\tilde{r}\geq \frac{T^{\delta/A}}{2\log T}$. Plugging this in, we get that 
$$
\int^{T^{\delta/A}/2}_{0} \chi_{\tilde{B}}(h_t(g_{\log(T^{1-\delta/A})}(x)))dt= \frac{1}{2}(1+o(1))T^{\delta/A}\mu(\tilde{B}).
$$
This implies that the LHS of \eqref{eq:chit} is $\geq (1-o(1))  T^{\delta/A}\log T \mu(\tilde{B})- \frac{1}{2}(1+o(1))T^{\delta/A}\mu(\tilde{B})>0$. 
\end{proof}

We also have the following result which is particularly useful for the proof of density of prime orbits. The proof of the result below uses the same methods as the proof above. We will provide a full proof for completeness. 

\begin{proposition}\label{prop:prime} Let $p\in X$ be a periodic point of period $R$. There is $i_0=i(R)\in \R_+$ such that for every $i\geq i_0$, every $\log T_i \leq H\leq T_i^{\alpha/10}$ there exists a set $G=G_{x,i,H}\subset [T_i/2,T_i\log T_i]$ with $|G|>T_iH^{-5}(\log T_i)^{-2}$ and such that for every $\bar{t}\in G$ there exists $\bar{\gamma}\in \Gamma$, $\bar{a}, \bar{c}$, such that 
\begin{equation}\label{eq:prime}
\bar{\gamma} x \begin{pmatrix}1&\bar{t}\\0&1 \end{pmatrix} =p\begin{pmatrix}\bar{a}&0\\\bar{c}&\bar{a}^{-1} \end{pmatrix},
\end{equation}
where $|\bar{a}-1|<H^{-1}$ and $|\bar{c}|<3H^{-2}$.
\end{proposition}
\begin{proof} Consider the set $\bar{V}_{T}(p)=\Big\{\Gamma p\begin{pmatrix}a&0\\c&a^{-1} \end{pmatrix}\;:\; |a-1|<H^{-1} \text{ and }|c|<3H^{-2}\}$. For simplicity we denote $T_i$ by $T$. We will apply Lemma \ref{lem:perp} for the time $H^4$. It implies that there exists $s\sim H^{-2+\delta/3}$ such that $g_{\log H^4}(v_s(p))\in B(e,1)$. Notice that $\{g_bv_{s+d}(p))\;:\; |b|<H^{-1}, |d|\leq H^{-4}\}\subset V_{T}(p)$. This just follows from the fact that 
$$
\begin{pmatrix}1&0\\s+d&1 \end{pmatrix}\cdot \begin{pmatrix}e^{b/2}&0\\0&e^{-b/2}\end{pmatrix}= \begin{pmatrix}e^{b/2}&0\\(s+d)e^{-b/2}&e^{-b/2}\end{pmatrix},
$$
and $(s+d)e^{-b/2}\leq  3H^{-2}$. So it is enough to  show that there exists $G$ as in the statement of the proposition such that for $\bar{t}\in G$, $h_{\bar{t}}(x)\in \{ g_b(v_dv_s(p))\;:\; |b|<H^{-1}, |d|\leq H^{-4}\}$. By composing both side of the last inclussion with $g_{\log(H^4)}$ and using the renormalization equations \eqref{eq:norm}, this in turn is equivalent to showing that there exists a set $\bar{G}\subset [H^{-4}T/2,H^{-4}T\log T] $  with $|\bar{G}|\geq \frac{1}{2}TH^{-9}$ and such that for $\bar{t}\in \bar{G}$, $h_{\bar{t}}(g_{\log(H^4)}x)\in \Big\{ g_b(v_dg_{\log(H^4)}v_s(p))\;:\; |b|<H^{-1}, |d|\leq 1\Big\}:=\bar{B}$. Note that $g_{\log(H^4)}v_s(p)\in B(e,1)$ by the definition of $s$. By transversality of $(h_t)$, it is enough to show the statement for $\tilde{B}=\bigcup_{u\in (0,1)}h_u(\bar{B})$. Notice that $\mu_{Haar}(\tilde{B})\sim H^{-1}$. 
Consider
\begin{equation}\label{eq:chit'}
\int^{H^{-4}T\log T}_{H^{-4}T/2} \chi_{\tilde{B}}(h_t(g_{\log(H^4)}(x)))dt.
\end{equation}
The proof now follows the same lines as the corresponding proof in Proposition \ref{prop:per}.
The above is equal to 
$$\int^{H^{-4}T\log T}_{0} \chi_{\tilde{B}}(h_t(g_{\log(H^4)}(x)))dt-\int^{H^{-4}T/2}_{0} \chi_{\tilde{B}}(h_t(g_{\log(H^4)}(x)))dt.
$$

By  Theorem \ref{thm:SFF}, we get that the main terms of the difference of both these integrals reduce to 
which 
$$
(H^{-4}T\log T-H^{-4}T/2)\mu(\tilde{B})\geq \frac{T}{H^5},
$$
Moreover the error for the first one is bounded by $H^{-4}(T\log T){\rm O}(H^5)r^{-\alpha}$, where $r=(H^{-4}T\log T)\cdot e^{-dist(g_{\log(H^{-4} T \log T)}(g_{\log(H^4)}(x)),e)}$. Note that $g_{\log(H^{-4} T \log T)}(g_{\log(H^4)}(x)\in K$ and so the error is  ${\rm O}(H^{5} T^{1-\alpha}\log T)= o(T/H^5)$ as $H<T^{\alpha/10}$. Analogously the error for the second one 
is bounded by  $H^{-4}T{\rm O}(H^5)r^{-\alpha}$, where $r=(H^{-4}T)\cdot e^{-dist(g_{\log(H^{-4} T)}(g_{\log(H^4)}(x)),e)}$. It remains to notice that $g_{\log T}x=g_{-\log \log T} g_{\log(T\log T)x}$ and $g_{\log(T\log T)x}\in K$. Plugging this to the error term shows that it is also $o(T/H^5)$. The statement follows. 
\end{proof}

The following lemma is straightforward: 
\begin{lemma}\label{lem:prime} Assume that $x,y\in X$ satisfy : for some $\gamma\in \Gamma$ and $|a-1|<H^{-1}$ $|c|<3H^{-2}$,
$$
\gamma x =y\begin{pmatrix}a&0\\c&a^{-1}\end{pmatrix}, 
$$
Then $d(h_t(x),h_t(y))<\epsilon$ for every $|t|<\epsilon^{2}H$ (here $H$ is large enough in terms of $\epsilon$).
\end{lemma}
\begin{proof} This is an immediate consequence of the identity
$$
\begin{pmatrix}1&t\\0&1\end{pmatrix}\begin{pmatrix}a&0\\c&a^{-1}\end{pmatrix}\begin{pmatrix}1&-t\\0&1\end{pmatrix}=\begin{pmatrix}a-ct&(a-a^{-1})t-ct^2\\c&a^{-1}+ct.\end{pmatrix}
$$
\end{proof}

Let $Orb(p)=\{h_tp\}_{t=0}^{R}$ be the orbit of the periodic point $p$. Proposition \ref{prop:per} has the following consequence:
\begin{lemma}\label{lem:per} Define $m_{a,c}(t)=\frac{at}{a^{-1}+ct}$. Let $a,c$ and $t_i\in [T/2,T\log (T)]$ be as in Proposition \ref{prop:per}. Then for every $t\in [t_i,t_i+T^{1/2-\delta/5}]$, 
$$
d\Big(h_t(x),h_{m_{a,c}(t-t_i)}(p)\Big)<\epsilon^{1/2}.
$$
Moreover, for any $t'\in [t_i,t_i+ T^{1/2-\delta/5}]$ there is $p'=p'(t')\in Orb(p)$  such that for any $t\in [t',t'+ T^{1/2-\delta/5}]$,
\begin{equation}\label{eq:apr}
d\Big(h_t(x),h_{m_{a+c(t'-t_i),c}(t-t')}(p')\Big)<\epsilon^{1/3}.
\end{equation}

\end{lemma} 
\begin{proof}
We will use Proposition \ref{prop:per}. Note that for $t\in [t_i,t_i+T^{1/2-\delta/5}]$, we have 
\begin{equation}\label{eq:gam}
\gamma x h_t=\gamma x h_{t_i} h_{t-t_i}=ph_{t-t_i}\Big[h_{-(t-t_i)}\begin{pmatrix}a&0\\c&a^{-1} \end{pmatrix}h_{t-t_i}\Big] .
\end{equation}
By a direct computation, 
$$
h_{-(t-t_i)}\begin{pmatrix}a&0\\c&a^{-1} \end{pmatrix}h_{t-t_i}= h_{m_{a,c}(t-t_i)-(t-t_i)} \begin{pmatrix}(a^{-1}+c(t-t_i))^{-1}&0\\c&a^{-1}+c(t-t_i) \end{pmatrix}
$$
Pluygging this into \eqref{eq:gam} we get
\begin{equation}\label{eq:comp}
\gamma x h_t=ph_{m_{a,c}(t-t_i)}\begin{pmatrix}(a^{-1}+c(t-t_i))^{-1}&0\\c&a^{-1}+c(t-t_i)\end{pmatrix}
\end{equation}
Notice that $|c(t-t_i)|\ll T^{-1/2+\delta/10}\cdot T^{1/2-\delta/5}=T^{-\delta/10}$ and $|a-1|<\epsilon$. Therefore the matrix on the RHS is $2\epsilon$ close to $Id$. The first part then follows from left invariance of $d$.

For the second part notice that by \eqref{eq:comp},
$$
\gamma x h_{t'}= p \cdot h_{m_{a,c}(t'-t_i)} \begin{pmatrix}(a^{-1}+c(t'-t_i))^{-1}&0\\c&a^{-1}+c(t'-t_i)\end{pmatrix}
$$
Denoting $p'=ph_{m_{a,c}(t'-t_i)}$, we get 
$$
\gamma x h_{t'}= p' \begin{pmatrix}(a^{-1}+c(t'-t_i))^{-1}&0\\c&a^{-1}+c(t'-t_i)\end{pmatrix}
$$
Let $A(t'):=\begin{pmatrix}(a^{-1}+c(t'-t_i))^{-1}&0\\c&a^{-1}+c(t'-t_i)\end{pmatrix}$. Then
$$
\gamma x h_t = \gamma x h_{t'} h_{t-t'}= p'h_{t-t'} [h_{-(t-t')}A(t')h_{t-t'}].
$$
This is an analogous equation to \eqref{eq:gam}. The matrix $A(t')$ has the diagonal entries of size  $1\pm 2\epsilon$ and $c$ did not change. We can therefore repeat the proof of the first part just with $t'$ instead $t_i$. This finishes the proof of the second part.
\end{proof}

\section{Density of orbits at times $\{n^{2-\delta}\}$; proof of Theorem \ref{thm:main}}
We will use the notation from the section above. Let $p\in X$ be a periodic point of period $R$ and let $\mu_p$ denote the unique invariant probability measure on $Orb(p)$. By the above results the proof of density of $\{h_{n^{2-\delta}}x\}$ follows from the following proposition:
\begin{proposition}\label{prop:equi} Let $f\in Lip(X)$ with compact support be fixed. There exists $t'\in [t_i,t_i+T^{1/2-\delta/5}]$  such that
$$
\sup_{p'\in Orb(p)}\Big|\frac{1}{|I|}\sum_{n\in I} f(h_{m_{a+c(t'-t_i),c}(n^{2-\delta}-t')}p')- \int_{Orb(p)}f(z)d\mu_p(z)\Big|=o_{T\to \infty }(1).
$$
where $I=[t'^{\frac{1}{2-\delta}},(t'+T^{1/2-\delta/5})^{\frac{1}{2-\delta}}]$.
\end{proposition}
The proof of Proposition \ref{prop:equi} will be given in a separate subsection, let us first show how it implies Theorem \ref{thm:main}.

\begin{proof}[Proof of Theorem \ref{thm:main}]
If the orbit $\{h_{n^{2-\delta}}x\}$ is not dense then there exists $\kappa>0$ and $y\in X$ such that the orbit never enters $B(y,\kappa)$. Let $f$  be a non-negative Lipschitz  function $f$ so that the support of $f$ is contained in $B(y,\kappa)$ (and positive on a fixed subset of $B(y,\kappa)$ of measure $1/2\mu(B(y,\kappa))$). It follows that for any interval $I\subset [0,\infty)$, 
$$
\sum_{n\in I} f(h_{n^{2-\delta}}x)=0.
$$
Let $I=[t'^{\frac{1}{2-\delta}},(t'+T^{1/2-\delta/5})^{\frac{1}{2-\delta}}]$ so that if $n\in I$ then $n^{2-\delta} \in [t',t'+T^{1/2-\delta/5}]$. Using Lemma \ref{lem:per} and Proposition \ref{prop:equi}  it follows that 
$$
\sum_{n\in I} f(h_{n^{2-\delta}}x)=\sum_{n\in I} f(h_{m_{a+c(t'-t_i),c}(n^{2-\delta}-t')}p')+{\rm O}_f(\epsilon^{1/3}|I|)= $$$$
|I|\int_{Orb(p)}f(z)d\mu_p(z)+o_{T\to \infty}(|I|)+ {\rm O}_f(\epsilon^{1/3}|I|).
$$
Note however that by a result of Sarnak, \cite{Sarnak}, it follows that $\int_{Orb(p)}f(z)d\mu_p(z)= \mu_{Haar}(f)+o_{R\to \infty}(1)$. Summarizing, 
$$
\sum_{n\in I} f(h_{n^{2-\delta}}x)\geq \frac{1}{2}|I|\mu_{Haar}(B(y,\kappa))+ o_{T\to \infty}(|I|)+ {\rm O}_f((\epsilon^{1/3})|I|)+o_{\min(R,T)\to \infty}(|I|).
$$
Taking $\epsilon$ small enough and $R,T$ large enough, we get that the LHS can not be $0$. This finishes the proof.
\end{proof}

\subsection{Proof of Proposition \ref{prop:equi}} 
We first start by approximating the function $m_{a,c}(\cdot)$ by polynomials. This is a simple consequence of Taylor expansion.

\begin{lemma}\label{lem:pol} Let $J=[0,T^{1/2-\delta/5}]$. Let $|a-1|<\epsilon$ and let $c\sim T^{-1/2+\delta/10}$.
 Let $b_i=b_i(a,c):=(-1)^{i-1}c^{i-1}a^{i+1}$ and let $P(t)=\sum_{i=1}^{d}b_it^i$ with $d=deg P=[5\delta^{-1}]+2$. Then
$$
|m_{a,c}(t)-P(t)|= {\rm O}\Big(\frac{1}{\log T}\Big) \text{ for every } t\in J.
$$
\end{lemma}
\begin{proof} Notice that we have for every $i\in \N$ $m^{(i)}_{a,c}(t)=\frac{(-1)^{i-1}i!c^{i-1}}{(a^{-1}+ct)^{i+1}}$. Notice that for $j=[5\delta^{-1}]+1$, by the bounds on $c$ and $a$,
$$
|m^{(j)}_{a,c}(t)||J|^j=\Big|\frac{(-1)^{j-1}j!c^{j-1}}{(a^{-1}+ct)^{j+1}}\Big| |J|^j\leq 4 j! T^{(j-1)(-1/2+\delta/10)+(1/2-\delta/5)j},
$$
and $(j-1)(-1/2+\delta/10)+(1/2-\delta/5)j<0$ by the choice of $j$. Therefore, $|m^{(j)}_{a,c}(t)||J|^j\leq (\log T)^{-1}$. So by Taylor's expansion up to order $j$:
$$m_{a,c}(t)=\sum_{i=0}m^{(i)}_{a,c}(0)\frac{1}{i!}t^i=\sum_{i=0}(-1)^{i-1} \frac{c^{i-1}}{a^{-(i+1)}}t^i+ {\rm O}\Big(\frac{1}{\log T}\Big).
$$
This finishes the proof.
\end{proof}

Let now $t'\in [t_i,t_i+T^{1/2-\delta/5}]$ be such that $t'^{\frac{1}{2-\delta}}\in \N$. Notice that such $t'$ exists as the length of the interval 
 $[t_i^{\frac{1}{2-\delta}},(t_i+T^{1/2-\delta/5})^{\frac{1}{2-\delta}}]$ is larger than $1$ (as can be seen by the mean value theorem for the function $x^{\frac{1}{2-\delta}})$.

 Consider the interval $I=I(t')=[t'^{\frac{1}{2-\delta}},(t'+T^{1/2-\delta/5})^{\frac{1}{2-\delta}}]$.  Since $t_i\in [T/2,T\log T]$, and $t'\in [t_i,2t_i]$ by the mean value theorem,
$$
 |I|\geq T^{1/2-\delta/5} \cdot (T\log T)^{\frac{1}{2-\delta}-1}=T^{\frac{\delta+2\delta^2}{10(2-\delta)}} (\log T)^{-1}.
 $$
 and analogously, $|I|\leq 4 T^{1/2-\delta/5} \cdot T^{\frac{1}{2-\delta}-1}\leq 4 T^{\frac{\delta+2\delta^2}{10(2-\delta)}}$.
 Summarizing
 \begin{equation}\label{eq:len}
  4 T^{\frac{\delta+2\delta^2}{10(2-\delta)}}\geq |I|\geq T^{\frac{\delta+2\delta^2}{10(2-\delta)}} (\log T)^{-1}
 \end{equation}

 Let $P_{t'}$ be the polynomial from Lemma \ref{lem:pol} applied to $\bar{a}=a+c(t'-t_i)$ and $c=c$ and let $b_i=b_i(t')=(-1)^{i-1}c^{i-1}\bar{a}^{i+1}$ be its coefficients.  
 Notice that if $n\in I$, then $|n^{2-\delta}-t'|<T^{1/2-\delta/5}$ and so by Lemma \ref{lem:pol} and smoothness of $(h_t)
$ we get  that 
$$
d(h_{m_{a+c(t'-t_i),c}(n^{2-\delta}-t')}p', h_{P_{t'}(n^{2-\delta}-t')}p')={\rm O}((\log T)^{-1}). 
$$

Therefore to prove Proposition \ref{prop:equi} it is enough to show the following:
\begin{equation}\label{eq:ssum}
\Big|\frac{1}{|I|}\sum_{n\in I} f(h_{P_{t'}(n^{2-\delta}-t')}p')- \int_{Orb(p)}f(z)d\mu_p(z)\Big|=o_{T\to \infty }(1).
\end{equation}
Note that by Taylor's expansion  for $x^{2-\delta}$ at $t'^{\frac{1}{2-\delta}}$ and the upper bound in \eqref{eq:len}, for $n\in I$,

\begin{equation}\label{eq:tay}
n^{2-\delta}-t'=(2-\delta)t'^{\frac{1-\delta}{2-\delta}}(n-t'^{\frac{1}{2-\delta}})+{\rm O}(t'^{\frac{-\delta}{2-\delta}}|I|^2)=(2-\delta)t'^{\frac{1-\delta}{2-\delta}}(n-t'^{\frac{1}{2-\delta}})+{\rm O}(T^{-\frac{\delta}{10(2-\delta)}}).
\end{equation}
We have the following straightforward lemma:
\begin{lemma}\label{lem:cl2} For any $n\in I$, we have 
$$
|P_{t'}\Big(n^{2-\delta}-t'\Big)-P_{t'}\Big((2-\delta)t'^{\frac{1-\delta}{2-\delta}}(n-t'^{\frac{1}{2-\delta}})\Big)|={\rm O}\Big(T^{-\frac{\delta}{10(2-\delta)}}\Big).
$$
\end{lemma}
\begin{proof} By the mean value theorem and by \eqref{eq:tay}, for some $\theta\in \Big[n^{2-\delta}-t',(2-\delta)t'^{\frac{1-\delta}{2-\delta}}(n-t'^{\frac{1}{2-\delta}})\Big]$  the LHS is $\leq |P'_{t'}(\theta)|{\rm O}(T^{-\frac{\delta}{10(2-\delta)}})$.
Moreover, again by \eqref{eq:tay}  it follows that $\theta\in  \Big[n^{2-\delta}-t',n^{2-\delta}-t'+1\Big]$. We claim that 
$$
|P'_{t'}(\theta)|={\rm O}(1).
$$
Indeed, by definition $P'_{t'}(\theta)=\sum_{i=1}^d (-1)^{i-1}i c^{i-1}\bar{a}^{i+1}\theta^{i-1}$. Moreover since $n\in I$, $\theta\leq n^{2-\delta}-t'+1\leq T^{1/2-\delta/5}+1$. The bound then follows as $|c|\sim T^{-1/2+\delta/10}$ and $\bar{a}=a+c(t'-t_i)\sim 1$ (since $t'\leq t_i+T^{1/2-\delta/5}$).
\end{proof}
By the above lemma, \eqref{eq:ssum} follows from the following:  
\begin{equation}\label{eq:ssum''}
\Big|\frac{1}{|I|}\sum_{n\in I} f(h_{P_{t'}(z(n,t'))}p')- \int_{Orb(p)}f(z)d\mu_p(z)\Big|=o_{T\to \infty }(1),
\end{equation}
where $z(n,t')=(2-\delta)t'^{\frac{1-\delta}{2-\delta}}(n-t'^{\frac{1}{2-\delta}})$.
We will now proof \eqref{eq:ssum''}. Before we prove it, we need some reductions which brings the analysis from $X$ to $S^1$ (using the fact that the point $p'$ is a periodic point of period $R$).

Let $L_t(\theta)=\theta+t \mod 1$ be the linear flow on the circle $S^1$. Let $\Delta:\{h_tp':0\leq t\leq R\}\to S^1$ be given by $\Delta(p')=0$ and $\Delta(h_tp')=L_{t/R}0$. Set $\tilde{f}:S^1\to \R$,
$\tilde{f}(x)=f(\Delta^{-1}x)$. Then since the map $\Delta$ is equivariant,
$$
\sum_{n\in I}f(h_{P_{t'}(z(n,t'))}p')=\sum_{n\in I}\tilde{f}(L_{\frac{P_{t'}(z(n,t'))}{R}}0).
$$
Since $\tilde{f}$ is a function on the circle it follows that $\tilde{f}(x)=\sum_{n\in \Z} a_ne_n(x)$, where $a_n=\int_{S^1}\tilde{f}(x)e_n(x)dx$. Note that $a_0=\int_{S^1}\tilde{f}(x)dx=\int_{Orb(p)}f(z)d\mu_p(z)$. Moreover since $f$ and $R$ are fixed, it follows that there exists $K=K(f,R,\epsilon)$ such that   $\Big|\tilde{f}(x)-\sum_{|k|<K} a_ke_k(x)\Big|<\epsilon$. 
Therefore, \eqref{eq:ssum''} (and hence also Proposition \ref{prop:equi}) follows immediately by the following lemma:

\begin{lemma}\label{lem:las}  For every $0 \neq |k|<K$, we have 
\begin{equation}\label{eq:ssum'}
\Big|\frac{1}{|I|}\sum_{n\in I}e\Big(\frac{k}{R}\cdot P_{t'}(z(n,t'))\Big) \Big|\leq \frac{\epsilon}{2K}.
\end{equation}
where $z(n,t')=(2-\delta)t'^{\frac{1-\delta}{2-\delta}}(n-t'^{\frac{1}{2-\delta}})$.
\end{lemma}
\begin{proof}[Proof of Lemma \ref{lem:las}] Notice that by the definition of $z(n,t')$ and $P_{t'}$, 
$$
\frac{k}{R}P_{t'}(z(n,t'))=\sum_{i=1}^{d} (-1)^{i-1}\Big((2-\delta)^i\frac{k}{R} c^{i-1}\bar{a}^{i+1}t'^{\frac{1-\delta}{2-\delta}i}\Big)(n-t'^{\frac{1}{2-\delta}})^i.
$$
 To simplify notation let  $\tilde{b}_{i}=(-1)^{i-1}\Big((2-\delta)^{i}c^{i-1}\bar{a}^{i+1}t'^{\frac{1-\delta}{2-\delta}i}\Big)$. Recall that \\$P_{t'}(\ell)=\sum_{i=1}^d (-1)^{i-1}c^{i-1}\bar{a}^{i+1}(\ell-t')^i$ for $\ell \in [t',t'+T^{1/2-\delta/5}]$. We claim that there exists $\eta=\eta(\delta)\in (0,\frac{\delta(2-\delta)}{10(2-\delta)}]$ and $2\leq i_0\leq d$ such that 
\begin{equation}\label{eq:gold}
C^{-1}(\log T)^{-i_0}T^{\eta}< |\tilde{b}_{i_0}| |I|^{i_0}<CT^{\eta}.
\end{equation}
Before we show \eqref{eq:gold}, let us show how to finish the proof using it. Notice in particular that \eqref{eq:gold} and \eqref{eq:len} (and the bound on $\eta$) imply that $|\tilde{b}_{i_0}| \ll T^{-\frac{\delta^2}{10(2-\delta)}}(\log T)^{i_0}$. Therefore if $T$ is sufficiently large, then for every $0<q\leq \log \log T$, every $|k|\leq K$ and every $R\leq \log T$
\begin{equation}\label{eq:ee}
\Big\|q\frac{k}{R}\tilde{b}_{i_0}\Big\|\geq  (\log T) |I|^{-i_0}
\end{equation}
Indeed, notice that by the bound on $|\tilde{b}_{i_0}|$ and $q,K,R$ it follows that $|q\frac{k}{R}\tilde{b}_{i_0}|<1/3$ and so 
$\Big\|q\frac{k}{R}\tilde{b}_{i_0}\Big\|= |q\frac{k}{R}\tilde{b}_{i_0}|$. By then by the lower bound in \eqref{eq:gold} it follows that 

$$\Big\|q\frac{k}{R}\tilde{b}_{i_0}\Big\|\geq \frac{C^{-1}(\log T)^{-i_0}T^{\eta}}{\log T}|I|^{-i_0}\geq (\log T) |I|^{-i_0}.$$ 
This gives \eqref{eq:ee}. Notice that then \eqref{eq:ssum'} follows from the Weyl bound, i.e. Lemma \ref{lem:Weyl}) for the polynomial $\frac{k}{R}\cdot P_{t'}(z(n,t'))$ where we use \eqref{eq:ee} to bound the $i_0$ coefficient. Here we pick  $T$ sufficiently large in terms of $K, \epsilon,d$ so that, say, $\log \log \log T>(\frac{\epsilon}{K})^{-O_d(1)}$. This finishes the proof.

So it only remains to prove \eqref{eq:gold}.
\begin{proof}[Proof of \eqref{eq:gold}]
Note that  by the bounds on $\bar{a}$, we have that there is $C'=C'(\delta,d)$ such that for every $i\leq d$,  
$\frac{1}{C'}c^{i-1} t'^{\frac{1-\delta}{2-\delta}i}\leq  |\tilde{b}_{i_0}|\leq C'|c|^{i-1} t'^{\frac{1-\delta}{2-\delta}i}$. Therefore it is enough to show that 
there is $i_0\leq d$ for which 
$$
C^{-1}(\log T)^{-i_0}T^{\eta}< c^{i_0-1} t'^{\frac{1-\delta}{2-\delta}i_0} |I|^{i_0}<CT^{\eta}
$$

Since $c\sim T^{-1/2+\delta/10}$ and $t'\in [T/2, T \log T]$ it follows by the lower bound in \eqref{eq:len} that for every $i\leq d$.
$$
c^{i-1} t'^{\frac{1-\delta}{2-\delta}i} |I|^{i}\gg T^{(-1/2+\delta/10)(i-1)} \cdot T^{\frac{1-\delta}{2-\delta}i} \cdot T^{(\frac{\delta+2\delta^2}{10(2-\delta)})i}(\log T)^{-i}= 
$$$$
(\log T)^{-i} \cdot T^{\frac{10-2\delta i+\delta^2 i-7\delta+\delta^2}{10(2-\delta)}}.
$$
Analogously, by the upper bound in \eqref{eq:len},
$$
c^{i-1} t'^{\frac{1-\delta}{2-\delta}i} |I|^{i}\ll T^{\frac{10-2\delta i+\delta^2 i-7\delta+\delta^2}{10(2-\delta)}}.
$$
Assume $\frac{10-7\delta+\delta^2}{\delta(2-\delta)}$ is not an integer. Then define $i_0:=\Big[\frac{10-7\delta+\delta^2}{\delta(2-\delta)}\Big]\geq 2$, and $\eta:= \frac{10-2\delta i_0+\delta^2 i_0-7\delta+\delta^2}{10(2-\delta)}$ . Note that 
$$
10-7\delta+\delta^2-i_0\delta(2-\delta)>0 
$$
and also
$$
10-7\delta+\delta^2-i_0\delta(2-\delta) \leq 10-7\delta+\delta^2-\Big(\frac{10-7\delta+\delta^2}{\delta(2-\delta)}-1\Big)\delta(2-\delta)= \delta(2-\delta),
$$
and so $\eta\in (0, \frac{\delta(2-\delta)}{10(2-\delta)}]$.
If  $\frac{10-7\delta+\delta^2}{\delta(2-\delta)}$ is an integer, we define $i_0:=\frac{10-7\delta+\delta^2}{\delta(2-\delta)}-1$ and $\eta:= \frac{10-2\delta i_0+\delta^2 i_0-7\delta+\delta^2}{10(2-\delta)}$. Note that in this case
$10-7\delta+\delta^2-i_0\delta(2-\delta)=\delta(2-\delta)$ and so $\eta=\frac{\delta(2-\delta)}{10(2-\delta)}$. In both cases we get that $\eta$ satisfies the bounds and so this finishes the proof of \eqref{eq:gold}.
This finishes the proof.
\end{proof}

\end{proof}

The following reformulation of Weyl's inequality can be found in \cite{GT}, Proposition 4.3:
\begin{lemma}\label{lem:Weyl} Let $g(n)=\sum_{i=1}^D \alpha_in^i$ be a polynomial of degree $D$.  If
$$
\Big|\sum_{n\in I}e(g(n-a))\Big|\geq \xi |I|
$$
for some integer interval $I=[a,b]$ and some $\xi\in (0,1/2)$, then there exists a positive integer $q\leq \xi^{-O_D(1)}$ such that 
$$
\|q\alpha_j\|\leq \xi^{-O_D(1)} |I|^{-j},
$$
for all $1\leq j\leq D$.
\end{lemma}

\section{Proof of Theorem \ref{thm:main'}}
In this section we will prove Theorem \ref{thm:main'}.  To do this we will show that for every $x\in X$ which is non-periodic, there is a sequence of intervals $I_i$ (with length going to $+\infty$), such that $\sum_{p\in I_i} f(h_px)=|I_i|\mu(f)(1+o(1))$. This is sufficient to deduce density by taking $f$ to be characteristic functions of balls. We will assume that $\mu(f)=0$.
We will need the following input from number theory.

\begin{proposition}\label{thm:EH}Fix $\alpha\in (0,1/2)$. Assume the Hardy-Littlewood conjecture in the form $\mathcal{C}_{10^6}(N^{1/50}, N)$ stated in section \ref{se:numbertheory}. For every $T>0$, every $H\sim T^{\alpha/30}$ and every $q\sim H^{1/4}$ the following holds: for all but a set $T^{1-\alpha/5-o(1)}$ of $x\in [0,T\log T]$ and for all but $o(q)$ of $a\in [0,q]\cap \Z$, we have 
\begin{equation}\label{eq:canc}
\sum_{n\in [x,x+H]}e\Big(n \frac{a}{q}\Big)\Lambda(n)=o(H).
\end{equation}
\end{proposition}
\begin{proof}
See section \ref{se:numbertheory}.
\end{proof}

For a fixed $x\in X$ and $\epsilon$ we will consider the sequence $\{T_i\}$ of times constructed in Proposition \ref{prop:prime}. For $T=T_i$ large enough and for $H'=\epsilon^{-2}H$ let $w$ be a periodic point of period $R=q/a$, where $a<q$ is such that $a, 2a,3a,...K(\epsilon)a$ are all good for the above theorem and $q/a\sim \epsilon^{-10}$ (existence of such $a$ follows from Pigeonhole principle). Consider the orbit $\{h_tx\}_{t\in [0,T\log T]}$ and let $G$ be the set coming from Proposition \ref{prop:prime}. Then it follows that for every $\bar{t}\in G$ we have that \eqref{eq:prime} holds. Then applying Lemma \ref{lem:prime} to  $h_{\bar{t}}x$ and $w$ it follows that  for every $t\in [\bar{t}, \bar{t}+H]$, $d(h_{t}x, h_{t-\bar{t}}w)<\epsilon$. In particular, 
$$
\sum_{n\in [\bar{t}, \bar{t}+H]}f(h_nx)\Lambda(n)=\sum_{n\in [\bar{t}, \bar{t}+H]}f(h_{n-\bar{t}}w)\Lambda(n).
$$
Now the proof is similar to the corresponding proof for polynomials $\{n^{2-\delta}\}$. We provide it here for completeness. 
Let $L_t(\theta)=\theta+t \mod 1$ be the linear flow on the circle $S^1$. Let $\Delta:\{h_tw:0\leq t\leq R\}\to S^1$ be given by $\Delta(w)=0$ and $\Delta(h_tw)=L_{t/R}0$. Set $\tilde{f}:S^1\to \R$,
$\tilde{f}(x)=f(\Delta^{-1}x)$. Then since the map $\Delta$ is equivariant,
$$
\sum_{n\in [\bar{t}, \bar{t}+H]}f(h_{n-\bar{t}}w)\Lambda(n)=\sum_{n\in [\bar{t}, \bar{t}+H]}\tilde{f}(L_{\frac{n-\bar{t}}{R}}0).
$$
Since $\tilde{f}$ is a function on the circle it follows that $\tilde{f}(x)=\sum_{n\in \Z} a_ne_n(x)$, where $a_n=\int_{S^1}\tilde{f}(x)e_n(x)dx$. Note that $a_0=\int_{S^1}\tilde{f}(x)dx=\int_{Orb(w)}f(z)d\mu_w(z)$, and if the period of $w$ goes to $\infty$, then $a_0=o(1)$.  Moreover since $f$ and $R$ are fixed, it follows that there exists $K=K(\epsilon)$ such that   $\Big|\tilde{f}(x)-\sum_{|k|<K} a_ke_k(x)\Big|<\epsilon$. Therefore to get that $\sum_{n\in [\bar{t}, \bar{t}+H]}f(h_{n-\bar{t}}w)\Lambda(n)=o(H)$, it is enough to show that 
$$
\sum_{n\in [\bar{t}, \bar{t}+H]}e(n\frac{ak}{q})\Lambda(n)= o(\frac{H}{K}).
$$
Note that $|G|>T^{1-\alpha/5}.$. Therefore there exists a $\bar{t}\in G$ such that $\bar{t}$ is  {\bf not} in the exceptional set from Theorem \ref{thm:EH}. In particular, by the choice of $a$ we get that \eqref{eq:canc} holds. This finishes the proof.

\section{Cancellation of short sums over primes} \label{se:numbertheory}

We make the following Hardy-Littlewood conjecture.
\begin{conjecture}
  We call $\mathcal{C}_{K}(H, N)$ the conjecture that uniformly in $1 \leq k \leq K$, $0 \leq x \leq N$ and distinct
  $1 \leq h_{i} \leq H$ we have,
  $$
  \sum_{n \leq x} \prod_{i = 1}^{k} \mathbf{1}_{n + h_{i} \in \mathbb{P}} = \mathfrak{S}(\mathcal{H}) \int_{1}^{x} \frac{dt}{(\log t)^{k}} + O_{K, \varepsilon} (N^{1/2 + \varepsilon}).
  $$
  where $\mathbb{P}$ denotes the set of prime numbers,
  $\mathcal{H} = \{h_{1}, \ldots, h_{k}\}$ and
  $$
  \mathfrak{S}(\mathcal{H}) = \prod_{p} \Big ( 1 - \frac{1}{p} \Big )^{-k} \Big ( 1 - \frac{\nu_{p}(\mathcal{H})}{p} \Big )
  $$
  and where $\nu_{p}(\mathcal{H})$ denotes the number of distinct residue classes modulo $p$ found in the set $\mathcal{H}$.
\end{conjecture}
For $H \leq N^{1/2}$ the conjecture $\mathcal{C}_{K}(H, N)$ is equivalent to the Hardy-Littlewood $k$-tuplet conjecture with von Mangoldt weight,
$$
\sum_{n \leq x} \prod_{i = 1}^{k} \Lambda(n + h_{i}) = \mathfrak{S}(\mathcal{H}) x + O(x^{{1/2 + \varepsilon}}).
$$
This version of the uniform Hardy-Littlewood conjecture is enunciated in \cite{SoundMontgomery}. This is just to say that conjectures such as $\mathcal{C}_{K}(H, N)$ with $H \leq N^{1/2}$ have some support in the literature.

Our main result is now the following.

\begin{proposition}
  Let $K \geq 1$. Assume conjecture $\mathcal{C}_{2K}(H, N)$.
  Let $(a,q) = 1$ with $1 < q < H^{1/2}$ and $\varepsilon, \delta > 0$. Then, 
  $$
  \Big | \sum_{x \leq n \leq x + H} \Lambda(n) e \Big ( \frac{n a}{q} \Big ) \Big | \leq \varepsilon H,
  $$
  for all $x \in [0, N]$ outside a set of measure $$\ll_{\delta, K} \varepsilon^{-2K} N H^{\delta} q^{-K/2} + \varepsilon^{-2K} \cdot N^{1/2 + \delta}.$$
\end{proposition}
In particular we see that the assumption of $\mathcal{C}_{10^6}(N^{1/50}, N)$ suffices to establish Proposition \ref{thm:EH}.

The above Proposition follows from applying Chebyschev's inequality and using the Proposition below.

\begin{proposition} \label{eq:newprop}
  Let $K \geq 1$. Assume $\mathcal{C}_{K}(H, N)$. Then, for $2k \leq K$ and $(a,q) = 1$ with $1 < q \leq H^{1/2}$, and $\varepsilon > 0$,
  $$
  \sum_{n \leq N} \Big | \sum_{n < p \leq n + H} e \Big ( \frac{p a}{q} \Big ) \Big |^{2k} \ll_{k, \varepsilon} N H^{2k + \varepsilon} q^{-k/2} + N^{1/2 + \varepsilon} H^{2k}.
  $$
\end{proposition}

Thus we focus on the proof of Proposition \ref{eq:newprop}.
We rewrite the expression in Proposition \ref{eq:newprop} as
$$
\sum_{n \leq N} \Big | \sum_{h \in [1, H]} e \Big ( \frac{(n + h) a}{q} \Big ) \mathbf{1}_{n + h \in \mathbb{P}} \Big |^{2k} = \sum_{n \leq N} \Big | \sum_{h \in [1, H]} e \Big ( \frac{h a}{q} \Big ) \mathbf{1}_{n + h \in \mathbb{P}} \Big |^{2k}.
$$
Expanding the $2k$th power we get,
$$
\sum_{1 \leq h_{1}, \ldots, h_{2k} \leq H} e \Big ( \frac{a}{q} \cdot (h_{1} + \ldots + h_{k} - h_{k + 1} - \ldots - h_{2k}) \Big ) \sum_{n \leq N} \prod_{i = 1}^{2k} \mathbf{1}_{n + h_{i} \in \mathbf{P}}.
$$
Let $\mathcal{D}_{1} \cup \ldots \cup \mathcal{D}_{r}  = \{1, 2, \ldots, 2k\}$ be a decomposition of $\{1,2,\ldots, 2k\}$ into disjoint subset. We split the above sum into subsums, so that whenever $i, j \in \mathcal{D}_{u}$ we set $h_i = h_j$ and for $i \in \mathcal{D}_{u}, j \in \mathcal{D}_{v}$ with $u \neq v$ we have $h_{i} \neq h_{j}$. Having fixed a decomposition we can rewrite the resulting sum as,
$$
\sum_{\substack{1 \leq h_{1}, \ldots h_{r} \leq H \\ \text{distinct}}} e \Big ( \sum_{i = 1}^{r} u_{i} h_{i} \frac{a}{q} \Big ) \sum_{n \leq N} \prod_{i = 1}^{r} \mathbf{1}_{n + h_{i} \in \mathbb{P}}.
$$
where we assume that the indices with $i \leq \ell$ appeared only once, and thus, $u_i = \pm 1$, while indices with $\ell < i$ have repetitions and thus $|u_i| \leq 2k$. Notice that a trivial bound for the above sum is $\ll H^{r}$ we can therefore assume $r > k$. Notice also that $\ell + 2 (r - \ell) \leq 2k$. Therefore
$$
2r - 2k \leq \ell.
$$
Now that we have constrained to the case of distinct $h_{i}$ we can appeal to the Hardy-Littlewood conjecture.
We find that the above is equal to,
$$
\int_{1}^{N} \frac{dt}{(\log t)^{r}} \cdot \sum_{\substack{1 \leq h_{1}, \ldots, h_{r}\leq H \\ \text{distinct}}} e \Big ( \frac{a}{q} \sum_{i = 1}^{r} u_{i} h_{i} \Big ) \mathfrak{S}(\{h_{1}, \ldots, h_{r}\}) + O(N^{1/2 + \varepsilon} H^{2k})
$$
where $\mathcal{H} := \{h_{1}, \ldots, h_{r}\}$.
Note that we can now remove the restriction to distinct $h_{i}$ since otherwise $\mathfrak{S}(\{h_{1}, \ldots, h_{r}\})$ vanishes.
Ideally we would aim for a bound such as,
$$
\sum_{1 \leq h_{1}, \ldots, h_{r} \leq H} e \Big ( \frac{a}{q} \sum_{i = 1}^{r} u_{i} h_{i} \Big ) \mathfrak{S}(\{h_{1}, \ldots, h_{r}\}) \ll_{\varepsilon}
H^{\ell / 2 + \varepsilon + (r - \ell)}
$$
which amounts to square root cancellation in the $h_{i}$ with $i \leq \ell$ and a trivial bound on the remaining indices for which $u_i = 0$ is a possibility. Since $\ell > 2 r - 2 k$ this would yield a bound that is,
$$
\ll H^{r - r + k + \varepsilon} = H^{k + \varepsilon}.
$$
We will in fact prove a weaker bound, namely, 
$$
\ll \Big ( \frac{H}{\sqrt{q}} \Big )^{\ell} \cdot H^{r - \ell}.
$$
Since $r \leq k + \ell / 2$ this gives, 
$$
\ll \Big ( \frac{H}{\sqrt{q}} \Big )^{\ell} \cdot H^{k - \ell / 2} \leq \Big ( \frac{H}{\sqrt{q}} \Big )^{\ell / 2} \cdot H^k \leq  \frac{H^{2k}}{q^{k/2}}. 
$$
where in the final inequality we used that $\ell \leq 2k$.

By \cite[equation (44)]{SoundMontgomery} we can write,
\begin{equation} \label{eq:newrep}
\mathfrak{S}(\mathcal{H}) = \sum_{\substack{q_{1}, \ldots, q_{r} \\ p | q_{i} \implies p \leq y}} \prod_{i = 1}^{r} \frac{\mu(q_{i})}{\varphi(q_{i})} A_{\mathcal{H}}(q_{1}, \ldots, q_{r}) + O \Big ( \frac{(\log y)^{r - 2}}{y} \Big ),
\end{equation}
where
$$
A_{\mathcal{H}}(q_{1}, \ldots, q_{r}) = \sum_{\substack{a_{1}, \ldots, a_{r} \\ 1 \leq a_{i} \leq q_{i} \\ (a_{i}, q_{i}) = 1 \\ \sum a_{i} / q_{i} \in \mathbb{Z}}} e \Big ( \frac{h_{i} a_{i}}{q_{i}} \Big )
$$
We therefore pick $y = H^{r}$ making a negligible error term.

Thus using the equation \eqref{eq:newrep} it remains to bound,
$$
\sum_{\substack{q_{1}, \ldots, q_{r} \\ p | q_{i} \implies p \leq H^{r}}} \prod_{i = 1}^{r} \frac{\mu(q_{i})}{\varphi(q_{i})} \sum_{\substack{1 \leq a_{i} \leq q_{i} \\ (a_{i}, q_{i}) = 1 \\ \sum a_{i} / q_{i} \in \mathbb{Z}}} \prod_{i = 1}^{r} \mathcal{D} \Big ( \frac{a_{i}}{q_{i}} + \frac{u_{i} a}{q} \Big ),
$$
where
$$
\mathcal{D} (\alpha) := \sum_{1 \leq h \leq H} e (h \alpha).
$$
To continue we will appeal to the following variant of a result of Montgomery-Vaughan.
\begin{lemma}
  Let $r_{1}, \ldots, r_{t}$ be squarefree integers, set $r = [r_{1}, \ldots, r_{t}]$. Suppose that for any prime dividing $r$ divides at least two of the $r_{i}$. Then, for any complex valued functions $G_{1}, \ldots, G_{t}$ defined on $(0, 1]$ we have,
  $$
\Big |  \sum_{\substack{b_{1}, \ldots, b_{t} \\ 1 \leq b_{i}\leq r_{i} \\ \sum b_{i} / r_{i} \in \mathbb{Z}}} \prod_{i = 1}^{t} G_{i} \Big ( \frac{b_{i}}{r_{i}} \Big ) \Big | \leq \frac{1}{r} \prod_{i = 1}^{t} \Big ( r_{i} \sum_{b_{i} = 1}^{r_{i}} |G_{i}(b_{i} / r_{i})|^{2} \Big )^{1/2}
$$
\end{lemma}
\begin{proof}
See \cite[Lemma 1]{SoundMontgomery}
\end{proof}

Notice that if $(b_{i}, r_{i}) = 1$ and
$$
\sum_{i = 1}^{t} \frac{b_{i}}{r_{i}} \in \mathbb{Z}
$$
then it is not possible for a prime $p$ to divide exactly one $r_{i}$, this is readily seen by clearing denominators.

Thus, we can apply the previous Lemma with the choice
$$
G_{i}(\alpha) := \mathbf{1}_{\alpha \in \mathcal{S}_{q_{i}}} \sum_{1 \leq h \leq H} e \Big ( h \alpha + \frac{h u_{i} a}{q} \Big )
  $$
  and
  $$
  \mathcal{S}_{q_{i}} := \Big \{ \frac{a}{q_{i}} : (a, q_{i}) = 1 \Big \}.
  $$
  Notice that $G_{i}$ is allowed to depend on $q_{i}$ which allows us to introduce the condition $\alpha \in \mathcal{S}_{q_i}$. Moreover when $\alpha = a_{i} / q_{i}$ the condition $a_{i} / q_{i} \in \mathcal{S}_{q_{i}}$ simply becomes the requirement that $(a_{i}, q_{i}) = 1$.

 We now come to bounding, 
  $$
  \sum_{1 \leq a_{i} \leq q_{i}} \Big | G_{i} \Big ( \frac{a_{i}}{q_{i}} \Big ) \Big |^{2}.
  $$
If $i > \ell$ we could have $u_i = 0$. Therefore for $i > \ell$ the best bound we can hope for is the trivial one, that is, 
 $$
  \sum_{1 \leq a_{i} \leq q_{i}} \Big | G_{i} \Big ( \frac{a_{i}}{q_{i}} \Big ) \Big |^{2} \leq q_i H^2.
  $$
For $i \leq \ell$ we have,
  $$
  \Big | G_{i} \Big ( \frac{a_{i}}{q_{i}} \Big ) \Big | \ll \min(\| a_{i} / q_{i} + a / q \|^{-1} , H )
  $$
  For $q_i \neq q$ this yields the bound,
  \begin{align*}
  \sum_{a_{i} = 1}^{q_{i}} \Big | G_{i} \Big ( \frac{a_{i}}{q_{i}} \Big ) \Big |^{2} & \ll H \sum_{a_i = 1}^{q_i} \Big | G_i \Big ( \frac{a_i}{q_i} \Big ) \Big | \\ & \ll H q q_i \log q_i \leq H \log H \cdot  q \cdot q_i^{1 - 1 / \log H}.
  \end{align*}
  Notice that in the above sum we can add the redundant condition $(a_i, q_i) = 1$, by definition of $G_i$. In the case $q_{i} = q$ one residue class could matches $a / q$. And therefore the final bound including this case, is,
  $$
 \sum_{a_i = 1}^{q_i} \Big | G_i \Big ( \frac{a_i}{q_i} \Big ) \Big |^{2} \ll H^{2} \mathbf{1}_{q = q_i} + H \log H \cdot q \cdot q_i^{1 - 1 / \log H}.
  $$
Since $q \leq H^{1/2}$ by assumptions, the above is, 
$$
\ll \frac{H^2 \log H}{q} \cdot q_i^{1 - 1 / \log H}.
$$
 It follows that, 
  \begin{align*}
\Big | & \sum_{\substack{1 \leq a_{i} \leq q_{i} \\ (a_{i}, q_{i}) = 1 \\ \sum a_{i} / q_{i} \in \mathbb{Z}}} \prod_{i = 1}^{r} \mathcal{D} \Big ( \frac{a_{i}}{q_{i}} + \frac{u_{i} a}{q} \Big ) \Big | \\ & \ll \frac{1}{[q_1, \ldots, q_r]} \prod_{i = 1}^{\ell} \Big ( \frac{H^2 \log H}{q}\cdot  q_i^{2 - 1 / \log H} \Big )^{1/2} \prod_{i = \ell + 1}^{r} \Big ( q_i^2 H^2 \Big )^{1/2}
  \end{align*}
We now sum this bound over $q_1, \ldots, q_r$ with weights $\varphi(q_i)^{-1}$ and over $q_i$ such that $p | q_i \implies p \leq H^r$. This yields, 
\begin{align*}
(H \log H)^r q^{-\ell / 2} \sum_{\substack{p | q_i \implies p \leq H^r}} \frac{q_1^{1 - 1 / \log H} \ldots q_r^{1 - 1 / \log H}}{[q_1, \ldots, q_r]} \cdot \frac{1}{\varphi(q_1) \ldots \varphi(q_r)}
\end{align*}
The sum over $q_1, \ldots, q_r$ is now an Euler product and bounded by $\ll (\log H)^{r^2}$. We conclude with the bound, 
$$
H^r (\log H)^{O_r(1)} q^{-\ell / 2},
$$
which is sufficient for our purposes.

\section{Equidistribution of periodic orbits at quadratic times}
\begin{proposition}
    Let $\phi$ be a smooth $\Gamma_0(1)$-invariant function on $\mathbb{H}$ (and compactly supported).
We have, 
$$
\sum_{n < q} \phi \Big ( \frac{n^2}{q} + \frac{i}{q} \Big ) \rightarrow \frac{1}{\text{Vol}(\mathcal{F})}\int_{\mathcal{F}} \phi(z) d \mu(z)
$$
as $q \rightarrow \infty$ along prime numbers congruent to $1$ modulo $4$, and where $\mathcal{F}$ is the fundamental domain. 
\end{proposition}
We make three remarks before proceeding to the proof. 
\begin{enumerate}
\item The restriction on $q$ to be prime is unnecessary and it's possible to prove a more general result asserting that, 
$$
\sum_{n < q} \phi \Big ( x + \frac{n^2}{q} + \frac{i}{q} \Big ) \rightarrow \frac{1}{\text{Vol}(\mathcal{F})} \int_{\mathcal{F}} \phi(z) d \mu(z)
$$
as $q \rightarrow \infty$ for any fixed $x \in \mathbb{R}$. The main difference is that we then need subconvexity for the additively twisted $L$-function, 
$$
\sum_{n \geq 1} \frac{\lambda_j(n)\chi_q(n) e(n x)}{n^s}
$$
which follows from the use of a circle method, but we couldn't locate an immediate reference, and where $\chi_q$ is the Kronecker symbol modulo $q$. 
\item Our proof gives in fact a rate of convergence that is $\ll q^{-\delta}$. 
\item Assuming the Generalized Riemann Hypothesis one can obtain a rate of convergence that is $\ll q^{-1/2 + \varepsilon}$ one could also diminish the length of the average over $n$ and shrink it to $n < q^{1/2 + \varepsilon}$.   
\end{enumerate}

\begin{proof}
We start by using the spectral expansion. We have, 
$$\phi(z) = \langle \phi, 1 \rangle + \sum_{j} \langle \phi, u_j \rangle u_j(z) + \int_{\mathbb{R}} \langle \phi, E(\cdot, \tfrac 12 + it) \rangle E(z, \tfrac 12 + it) dt $$
with convergence being uniform and absolute for all $z$ in any given compact set and where $u_j$ are eigenfunctions of the hyperbolic Laplacian in the discrete spectrum and with eigenvalues $\tfrac 14 + t_j^2$ and $E(z, \tfrac 12 + it)$ are in the continuous spectrum with eigenvalues $\tfrac 14 + t^2$. 
By integration by parts we have, 
$$
\langle \phi, u_j \rangle \ll_{A} (\tfrac 14 + t_j^2)^{-A}
$$
for any $A > 10$ and where $\tfrac 14 + t_j^2$ is the eigenvalue of the eigenfunction $u_j$. Similarly, 
$$
\langle \phi, E(\cdot, \tfrac 12 + it) \rangle \ll_{A} (\tfrac 14 + t^2)^{-A}
$$
for any given $A > 10$. Therefore we can truncate both the sum and the integral at $|t_j| \leq q^{\varepsilon}$ and $|t| \leq q^{\varepsilon}$ respectively, making an error that is $\ll_{A, \varepsilon} q^{-A}$.

Therefore it remains to show that there exists a $\delta > 0$ such that, 
$$
\sum_{n < q} u_j\Big ( \frac{n^2}{q} + \frac{i}{q} \Big ) \ll q^{1 - \delta}
$$
and $$
\sum_{n < q} E \Big ( \frac{n^2}{q} + \frac{i}{q} \Big ) \ll q^{1 - \delta}.
$$
Both estimates are in fact similar. 
\subsubsection{The contribution of the cuspidal spectrum}
We recall that the Fourier expansion of a cusp form of eigenvalue $\tfrac 14 + t_j^2$, normalized so that $\| u_j \| = 1$ is given as, 
$$
u_j(z) = c(t_j) \sqrt{y} \sum_{n \geq 1} \lambda_j(n)\kappa_{i t_j}(2\pi |n| y) \text{SC}(n x)
$$
where $\text{SC}(x)$ is either the $\sin$ or $\cos$ function and
$\kappa_{i t_j} = e^{\frac{\pi}{2} t_j} K_{i t_j}(x)$ where $K(\cdot)$ is the $K$-Bessel function, and finally $c(t_j) = t_j^{o(1)}$

Upon averaging we thus find, 
$$
\sum_{n < q} u_j \Big ( \frac{n^2}{q} + \frac{i}{q} \Big ) = c(t_j) \cdot \frac{1}{\sqrt{q}} \sum_{n \geq 1} \lambda_j(n) \kappa_{i t_j} \Big ( \frac{2\pi n}{q} \Big ) \Big ( \sum_{x \pmod{q}} \text{SC} \Big (\frac{2\pi n x^2}{q} \Big ) \Big ).
$$
We notice that the exponential sum is zero if $\text{SC} = \sin$ and evaluates to a standard Gauss sum when $\text{SC} = \cos$. 
The Gauss sum now evaluates, 
$$
\sum_{x \pmod{q}} \cos \Big ( \frac{2\pi n x^2}{q} \Big ) = \sqrt{q} \Big ( \frac{n}{q} \Big )
$$
Furthermore since $q$ is a prime congruent to $1$ modulo $4$ we have, 
$$
\Big ( \frac{n}{q} \Big ) = \chi_{q}(n)
$$
where $\chi_q$ is the Kronecker symbol modulo $q$. 
We recall that, for $\Re s > 0$, 
$$
\int_{0}^{\infty} \kappa_{i t_j}(x) x^{s - 1} dx = 2^{s - 2} e^{\frac{\pi}{2} t_j} \Gamma \Big ( \frac{s - i t_j}{2} \Big ) \Gamma \Big ( \frac{s + i t_j}{2} \Big ). 
$$
Therefore, by Mellin inversion, 
\begin{align*}
\sum_{n \geq 1} & \lambda_j(n) \kappa_{i t_j} \Big ( \frac{2\pi |n|}{q} \Big ) \chi_q(n) \\ & = \frac{1}{8\pi i} \int_{1/2 - i \infty}^{1/2 + i \infty} (4\pi q)^{s} L(s, u_j \otimes \chi_q) e^{\frac{\pi}{2} t_j} \Gamma \Big ( \frac{s + i t_j}{2} \Big ) \Gamma \Big ( \frac{s - i t_j}{2} \Big ) ds.  
\end{align*}
We now appeal to a subconvexity bound originally due to Iwaniec (see \cite{Young} for the best current result), 
$$
L(\tfrac 12 + it, u_j \otimes \chi_q) \ll (t_j q)^{1/2 - \delta}
$$
to conclude that the above is $\ll q^{1 - \delta + o(1)}$ as needed.

\subsubsection{The contribution of the Eisenstein spectrum}

This is very much similar to the contribution of the cuspidal spectrum. We therefore only highlight the main difference. We need to show that, 
$$
\sum_{n < q} E \Big ( \frac{n^2}{q} + \frac{i}{q}, \frac{1}{2} + it \Big ) \ll q^{1 - \delta} 
$$
for some $\delta > 0$ and $|t| \leq q^{\varepsilon}$. We use the Fourier expansion of an Eisenstein series, 
\begin{align*}
E(z, \tfrac 12 + it) = y^{\tfrac 12 + it} & + \frac{\theta(\tfrac 12 + it)}{\theta(\tfrac 12 - it)} y^{1/2 - it} \\ & + \frac{4 \sqrt{y}}{\theta(\tfrac 12 + it)} \sum_{n > 0} \eta_{it}(n) K_{it}(2\pi n y) \cos(2\pi n x)
\end{align*}
where $\theta(s) = \pi^{-s}\Gamma(s)\zeta(2s)$
and 
$$
\eta_{it}(n) = \sum_{a b = n} \Big ( \frac{a}{b} \Big )^{it}.
$$
The contribution of the terms $y^{1/2 \pm it}$ is negligible and gives a bound of $\sqrt{q}$. 
Proceeding as before we are led to consider,
$$
\frac{1}{2\pi i} \int_{1/2 - i \infty}^{1/2 + i \infty} (2\pi q)^{s} L(s + it, \chi_q) L(s - it, \chi_q) \Gamma \Big ( \frac{s - i t}{2} \Big ) \Gamma \Big ( \frac{s + it}{2} \Big ) ds 
$$
and this is again $\ll q^{1 - \delta}$ by using the subconvex bound \cite{Young}, 
$$
L(\tfrac 12 + it, \chi_q) \ll ((1 + |t|) q)^{1/4 - \delta}.
$$

\end{proof}


\begin{thebibliography}{9}
\bibitem{Dani} S. G. Dani, {\em Invariant measures of horospherical flows on noncompact homogeneous spaces}, Invent. Math., 47 (1978), 101--138.
\bibitem{DS} S. G.\ Dani, J. Smillie, {\em Uniform distribution of horocycle orbits for Fuchsian groups}, Duke Math. J., 51(1): (1984), 185--194.
\bibitem{Flaminio-Forni} L.\ Flaminio, G.\ Forni,{\em Invariant distributions and time averages for horocycle flows}, Duke
Math. J., 119 (2003), 465--526.
\bibitem{FFT} L. Flaminio, G. Forni, J. Tanis,{\em Effective equidistribution of twisted horocycle flows and horocycle maps}, Geometric and Functional Analysis, 26(5):1359--1448, 2016.
\bibitem{FKR} G.\ Forni, A.\ Kanigowski, M.\ Radziwi\l\l, {\em Horocycle flow at product of two primes}, arXiv:2409.16687
\bibitem{Fur} H.\ Furstenberg,{\em The unique ergodicity of the horocycle flow},  in Recent Advances in Topological Dynamics (Proc. Conf., Yale Univ., New Haven, CT, 1972), Lecture Notes in Math., 318, Springer, Berlin, 1973, 95--115.
\bibitem{GT} B.\ Green, T.\ Tao, {\em The quantitative behaviour of polynomial orbits on nilmanifolds}, Ann. of Math. (2), 175(2): (2012), 465--540.
\bibitem{Howe}  R. Howe,{\em On a notion of rank for unitary representations of the classical groups}, Harmonic analysis and group representations, 223- 331, Liguori, Naples, 1982.
\bibitem{LMW}  E.\ Lindenstrauss, A.\ Mohammadi, Z.\ Wang,{\em Effective equidistribution for some one parameter unipo-
tent flows}, arXiv:2211.11099.
\bibitem{McAdam} T.\ McAdam,{\em Almost-prime times in horospherical flows on the space of lattices}, J. Mod. Dyn. 15
(2019), 277–327.
\bibitem{SoundMontgomery} H. L. Montgomery, K. Soundararajan, {\em Primes in short intervals}, Communications in mathematical physics \textbf{252} (1), 589-617
\bibitem{Moore} C.\ Moore,{\em Exponential decay of correlation coefficients for geodesic flows}, Group representations, ergodic theory, operator algebras, and mathematical physics (Berkeley, Calif., 1984), 163-181, Math. Sci. Res. Inst. Publ., 6, Springer, New York, 1987.
\bibitem{Par} O. S. Parasyuk, {\em Flows of horocycles on surfaces of constant negative curvature (in Russian)}, Uspekhi Mat. Nauk 8, no. 3, 1953, 125--126.
\bibitem{Ratner} M.\ Ratner, {\em Raghunatan's conjectures for SL(2,R)}, Israel J. Math., 80 (1992), 1--31.
\bibitem{Sarnak} P.\ Sarnak, {\em Asymptotic behavior of periodic orbits of the horocycle flow and Eisenstein series},
Comm. Pure Appl. Math., 34 (1981), 719--739.
\bibitem{SU} P.\ Sarnak, A.\ Ubis,{\em The horocycle flow at prime times}, J. Math. Pures Appl. (9) 103 (2015), 575--618.
\bibitem{Stre} L. Streck,{\em Non-Concentration of Primes in Γ\ P SL(2, R)}, preprint, arXiv:2303.07781v1
\bibitem{Strombergsson} A.\ Str\"ombergsson,{\em On the deviation of ergodic averages for horocycle flows}, Journal of Modern Dynamics, 7 (2013), 291--328.
\bibitem{Ven} A. Venkatesh,{\em Sparse equidistribution problems, period bounds and subconvexity}, Ann. of Math. (2), 172 (2010), 989--1094.
\bibitem{Young} M. P. Young, {\em Weyl-type hybrid subconvexity bounds for twisted L-functions and Heegner points on shrinking sets}, J. Eur. Math. Soc. 19 (2017), no. 5, pp. 1545--1576
\end{thebibliography}
\end{document}